\documentclass[a4paper]{article}

\usepackage{graphics}
\usepackage{graphicx}
\usepackage{amsfonts}
\usepackage{amssymb}
\usepackage{amsmath}
\usepackage{subfigure}
\usepackage{color}
\usepackage[affil-it]{authblk}

\usepackage{graphicx}
\usepackage{amsthm}
\usepackage{mathtools}
\usepackage[subnum]{cases}

\newcounter{oftheorem}[subsection]
\newenvironment{mytheorem}[1]%
{\begin{trivlist}
     \renewcommand{\theoftheorem}{#1~\thesection.\arabic{oftheorem}}
     \refstepcounter{oftheorem}
     \item[\hspace{\labelsep}\bf\thesection.\arabic{oftheorem} #1]}%
{\end{trivlist}}

\newenvironment{proposition}{\begin{mytheorem}{Proposition}\it}{\end{mytheorem}}

\newenvironment{corollary}{\begin{mytheorem}{Corollary}\it}{\end{mytheorem}}

\newenvironment{lemma}{\begin{mytheorem}{Lemma}}{\end{mytheorem}}

\begin{document}

\title{\textbf{Stationary Solitons in discrete NLS with
non-nearest neighbour interactions}}

\author{Vassilis M.~ Rothos$^\S$, Stavros Anastassiou$^\dag$ and Katerina G. Hadjifotinou$^\ddag$}

\affil{\small $^\S$ Department of Mechanical Engineering, Faculty of Engineering, Aristotle University of
Thessaloniki, Thessaloniki 54124, Greece}
\affil{\small $^\dag$ Department of Mathematics, University of West Macedonia, Kastoria 52100, Greece}
\affil{\small $^\ddag$ Department of Mathematics,Faculty of Science, Aristotle University of
Thessaloniki, Thessaloniki 54124, Greece}

\maketitle

\begin{abstract}
The aim of this paper is to provide a construction of stationary discrete solitons
in an extended one-dimensional Discrete NLS  model with non-nearest neighbour
interactions. These models, models of the type with long-range interactions were studied in various other contexts. In
particular, it was shown that, if the interaction strength decays sufficiently slowly as a function of distance, it gives
rise to bistability of solitons, which may find applications in their controllable switching. Dynamical lattices
with long-range interactions also serve as models for energy and charge transport in biological molecules. Using a dynamical systems method we are able to construct, with great accuracy, stationary discrete solitons for our model, for a large region of the parameter space.
\end{abstract}

\section{Introduction}

Among the most prominent differential equations in applied mathematics are nonlinear
Schr\"{o}dinger (NLS) equations, a class of partial differential equations which have been employed as
models in diverse branches of physics. The nonlinear Schr\"{o}dinger (NLS) equation describes a very large variety of physical systems since it is the lowest order nonlinear (cubic) partial differential equation that describes the propagation of modulated waves. For example, two of the most salient applications of NLS equations, that will be the main theme of this session, stem from the realm on nonlinear optics and Bose-Einstein condensates (BECs). In optics, the NLS stems from the nonlinear (Kerr) response of the refractive index of some nonlinear media. On the other hand, for BECs, the mean-field description of a condensed cloud of atoms is well approximate, at low enough temperatures, by the NLS with an external potential. 

It is well known that the solitons described by the regular NLS equations with power-law nonlinearities may be unstable if the number of space dimensions or (and) the power of nonlinearity is sufficiently high \cite{4th_1}, \cite{4th_2}. In some cases the instabilities result in the collapse (blowup) which occurs at finite time (in two space dimensions this is the so-called self-focusing). In many papers, the main attention in the studies of stabilizations of instabilities has been paid to the saturation of nonlinearity. On the other hand, the effects of higher-order dispersion may also play a significant role \cite{4th_3}. In this work, we continue the investigation of the stabilizing role of higher-order dispersive effects using a comparatively simple model described by the equation
\begin{equation}
\label{1.1}
\mathrm{i} \partial_t \Psi+\frac{1}{2} \Delta \Psi+\frac{1}{2} \gamma \Delta^2 \Psi+|\Psi|^{2 p} \Psi=0,
\end{equation}
where $p$ is an integer number, and
$$
p \geq 1, \quad \Delta=\nabla_\alpha \nabla_\alpha, \quad \alpha=1, \ldots, D, \quad D=1,2,3
$$
(summation over repeated indices is assumed; $D$ is the number of space dimensions). The term containing $\Delta^2$ describes the higher-order dispersion. Eq. \eqref{1.1} permits one to investigate the role of space dimensions, strength of  nonlinearity and higher-order dispersion effects in the soliton stabilization. The criteria for soliton stability, briefly reported earlier \cite{4th_3}, contain $D$ and $p$ not separately, but through the product $D p$; at $\gamma=0$ they turn into the conditions derived for the NLS type equations with the lowest order dispersion \cite{4th_1}, \cite{4th_2}. Fourth-order Schr\"{o}dinger equations have been introduced to take into account the role of small fourth-order dispersion terms in the propagation of intense laser beams in a bulk medium with Kerr nonlinearity. Such fourth-order Schr\"{o}dinger equations have been studied from the mathematical viewpoint in Fibich, Ilan and Papanicolaou \cite{4th_4} who describe various properties of the equation in the subcritical regime, with part of their analysis relying on very interesting numerical developments. 

During the early years, studies of intrinsic localized modes were
mostly of a mathematical nature, but the ideas of localized modes
soon spread to theoretical models of many different physical
systems, and the discrete breather concept has been recently
applied to experiments in several different physics
subdisciplines. Most nonlinear lattice systems are not integrable
even if the partial differential equation (PDE) model in the
continuum limit is. While for many years spatially continuous
nonlinear PDE's and their localized solutions have received a
great deal of attention, there has been increasing interest in
spatially discrete nonlinear systems. Namely, the dynamical
properties of nonlinear systems based on the interplay between
discreteness, nonlinearity and dispersion (or diffraction) can
find wide applications in various physical, biological and
technological problems. Examples are coupled optical fibres
(self-trapping of light)~\cite{Cris,Sergej,coupled,Lenz}, arrays
of coupled Josephson junctions~\cite{Joseph}, nonlinear charge and
excitation transport in biological macromolecules, charge
transport in organic semiconductors~\cite{semi}.

Prototype models for such nonlinear lattices take the form of
various nonlinear lattices \cite{Aubry}, a particularly important
class of solutions of which are the so called discrete breathers which
are homoclinic in space and oscillatory in time. Other questions
involve the existence and propagation of topological defects or
kinks which mathematically are heteroclinic connections between a
ground and an excited steady state. Prototype models here are
discrete version of sine-Gordon equations, also known as
Frenkel-Kontorova (FK) models, e.g.\ \cite{ACR03}. There are many
outstanding issues for such systems relating to the global
existence and dynamics of localized modes for general
nonlinearities, away from either continuum or anti-continuum
limits.

In the main part of the previous studies of the discrete NLS
models the dispersive interaction was assumed to be short-ranged
and a nearest-neighbour approximation was used. However, there exist
physical situations that definitely can not be described in the
framework of this approximation. 

The existence of localized traveling waves (sometimes called ``moving
discrete breathers'' or ``discrete solitons'') in discrete nonlinear
Schr\"{o}dinger (DNLS) lattices has shown itself to be a delicate
question of fundamental scientific interest (see e.g \cite{FK99}).  
This interest
is largely due to the experimental realization of solitons in discrete
media, such as waveguide arrays \cite{CLS03} optically induced
photorefractive crystals \cite{FCSEC03} and Bose-Einstein condensates
coupled to an optical lattice trap \cite{KRB01}. The prototypical
equation that emerges to explain the experimental observations is the DNLS model of the form:
\begin{equation}
\label{eqn:dnls}
i \dot{u}_{n}(t) = \frac{u_{n+1}(t) -
2u_{n}(t) + u_{n-1}(t)}{h^{2}}
+ F(u_{n+1}(t), u_{n}(t), u_{n-1}(t)),
\end{equation}
where the integer $n \in \mathbb{Z}$ labels a one-dimensional array of lattice
sites, with spacing  $h$. Alternatively $h$ can be thought of as
representing the inverse coupling
strength between adjacent sites.
The nonlinear term $F$ can take a number of different
forms:
\begin{itemize}
\item {DNLS equation: $F_{\rm DNLS} = |u_n|^2 u_n$,
}
\item {Abolowitz-Ladik (AL) model \cite{AL76}: $F_{\rm AL}= |u_n|^2(u_{n+1} + u_{n-1})$,}
\item {Salerno model \cite{S92}: $F_{\rm S}=2(1-\alpha) F_{\rm DNLS} + \alpha F_{\rm AL}$,}
\item {cubic-quintic DNLS: $F_{3-5}=(|u_n|^2 + \alpha |u_n|^4)u_n
$,}
\item {saturable DNLS: $F_{\rm sat}= \frac{u_n}{1+|u_n|^2}$,}
\item {generalised cubic DNLS equation:
\begin{eqnarray}
\nonumber F_{g3} & = & \alpha_1 |u_n|^2 u_n + \alpha_2 |u_n|^2 (u_{n+1} +
u_{n-1}) + \alpha_3 u_n^2 (\bar{u}_{n+1} + \bar{u}_{n-1}) \\
\nonumber & \phantom{t} & \alpha_4 (|u_{n+1}|^2 + |u_{n-1}|^2) u_n
+ \alpha_5 (\bar{u}_{n+1} u_{n-1} + u_{n+1} \bar{u}_{n-1} ) u_n + \\
\nonumber & \phantom{t} & \alpha_6 (u_{n+1}^2 + u_{n-1}^2) \bar{u}_n +
\alpha_7 u_{n+1} u_{n-1} \bar{u}_n \\
\nonumber & \phantom{t} & + \alpha_8 (|u_{n+1}|^2 u_{n+1} +
|u_{n-1}|^2 u_{n-1}) + \alpha_9 (u_{n+1}^2 \bar{u}_{n-1} + u_{n-1}^2 \bar{u}_{n+1} ) \\
& \phantom{t} & + \alpha_{10} (|u_{n+1}|^2 u_{n-1} + |u_{n-1}|^2
u_{n+1}). \label{eqn:Dimitry}
\end{eqnarray}
where $\bar{\cdot}$ is used to represent complex conjugation.}
\end{itemize}

Note that when $\alpha_1 = 2(1 - \alpha_2)$, $\alpha_2 \in \mathbb{R}$, and
$\alpha_j = 0$ for $3 \leq j \leq 10$, the nonlinear function $F_{g3}$ reduces
to the Salerno nonlinearity $F_{S}$. 
Stationary localized solutions to \eqref{eqn:dnls} of
the form $u(n,t) = e^{-i \omega t} U(n)$
abound in such models under quite general hypotheses on the function $F$,
and indeed one can pass to the continuum limit $h \to 0$, $x=nh$
and find the corresponding solutions to continuum NLS equations of the form
\begin{equation}
i \dot{u}_{n}(t) = u_{xx} + f( |u|^2) u.
\label{eqn:nls}
\end{equation}

The goal of this work is to study the existence of special solutions (stationary solutions) for extended nonlinear lattice models of the 4th order NLS-type equation \eqref{1.1}. 


We consider the DNLS-type model with the next-nearest-neighbor interaction for simplicity of our analysis (we could consider the direct discretization of \eqref{1.1}, in a future publication)

\begin{equation}
\label{mainpdeeq}
i\dot{u}_n+|u_n|^2u_n+\epsilon \Big(u_{n+1}-2u_{n}+u_{n-1}+ A\big(u_{n+2}+u_{n-2}\big)\Big)=0,
\end{equation}
where $u_n:\mathbb{R}_+\rightarrow \mathbb{C},\ n\in \mathbb{Z},\ \epsilon >0$.

We seek steady--state solutions of this equation, i.e. solutions that satisfy $\dot{u}_n(t)=0,\forall n\in \mathbb{Z}$.

We begin our study in section 2, where we set $A=0$ in equation (\ref{mainpdeeq}). A two--dimensional mapping is obtained and its qualitative behaviour is studied, which turns out to be quite simple.

In section 3, we turn our attention to the case where $A\neq 0$. In this case, by setting $\dot{u}_n(t)=0,\forall n\in \mathbb{Z}$, a four--dimensional mapping is obtained, the qualitative properties of which are of interest here. Its symmetry properties are analysed and the stability of its fixed points is determined. We then proceed to compute orbits homoclinic to the origin for this, four--dimensional, mapping. To achieve that, we employ the Parametrization Method and are able to locate, with great accuracy, homoclinic points when parameters $A,\ \epsilon$ belong to a certain region of the parameter space. We also argue that these homoclinic points are transverse, implying the existence of complicated behaviour in the phase space of our mapping. Note that, in this way, we have constructed stationary soliton solutions of equation (\ref{mainpdeeq}). Last section contains the conclusions of this work.  


\section{The 2--d mapping}
To study steady--state solutions for equation (\ref{mainpdeeq}), we let $\dot{u}_n=0$. We begin, in this section, be setting $A=0$. We get:
\[
|u_n|^2u_n+\epsilon \big(u_{n+1}-2u_{n}+u_{n-1}\big)=0\Rightarrow u_{n+1}=-u_{n-1}+2u_n-\frac{1}{\epsilon}|u_n|^2u_n.
\]
By introducing variables $x=u_{n-1},\ y=u_n$, we arrive at the following mapping of the complex plane:
\begin{equation}
f_0:\mathbb{C}^2\rightarrow \mathbb{C}^2,\ f_0(x,y)=(y,-x+2y-\frac{1}{\epsilon}|y|^2y).
\end{equation}
Observe that $f_0(x,y)\in \mathbb{R}^2$ if, and only if, $(x,y)\in \mathbb{R}^2$. We therefore restrict our attention to the case where the variables are real, and consider the following mapping, which we also denote by $f_0$:
\begin{equation}
\label{2dmap}
f_0:\mathbb{R}^2\rightarrow \mathbb{R}^2,\ f_0(x,y)=(y,-x+2y-\frac{1}{\epsilon}y^3).
\end{equation}
In this section, it is our intention to study mapping (\ref{2dmap}).

Observe that it is a generalised H\'{e}non map, i.e., a mapping of the form:
\[
(x,y)\mapsto (y,\delta x+p(y)),\]
which has been extensively studied, not only in its two--dimensional form presented above, but in its higher dimensional analogues as well (see, for example, \cite{LM,SDM,DM,GOST,GMO,GLM,Z,GG,ABB1,ABB2,A} and references therein).   

The reader can easily confirm that, $\forall (x,y)\in \mathbb{R}^2,\ |Df_0(x,y)|=1$, thus our mapping is area--preserving, while its inverse is given by:
\[
f_0^{-1}:\mathbb{R}^2\rightarrow \mathbb{R}^2,\ f_0^{-1}(x,y)=(2 x-\frac{1}{\epsilon}x^3-y,x).
\]
Moreover, mapping $f_0$ is symmetric with respect to the involution $\sigma_1(x,y)=(-x,-y)$, i.e. relation $f_0\circ \sigma_1 =\sigma_1 \circ f_0$ holds, while it is reversible with respect to the involutions $\sigma_2(x,y)=(y,x),\ \sigma_3(x,y)=(-y,-x)$, that is, 
\[
f_0\circ \sigma_2=\sigma_2\circ f_0^{-1},\ f_0\circ \sigma_3=\sigma_3\circ f_0^{-1}.
\]

Let us locate bounded orbits for mapping (\ref{2dmap}).
\begin{proposition}
Mapping (\ref{2dmap}) possesses, $\forall \epsilon \in \mathbb{R}$, a unique fixed point, located at the origin, which is a degenerate parabolic point. This fixed point is stable if, and only if, $\epsilon >0$. The non--wandering set of the mapping is contained in $\big[-2\sqrt{|\epsilon|},2\sqrt{|\epsilon|}\big]\times \big[-2\sqrt{|\epsilon}|,2\sqrt{|\epsilon|}\big]\subset \mathbb{R}^2$.
\end{proposition}
\begin{proof}
Simple calculations show that, indeed, the origin is the unique fixed point of mapping $f_0$, having a double eigenvalue equal to $1$.

Consider the change of variables $\psi(x,y)=(x+y,2x+y)$. It is area--preserving and satisfies relation $f_0\circ \psi=\psi \circ T$, where:
\[
T(x,y)=(x-\frac{1}{\epsilon}(2x+y)^3,x+y+\frac{1}{\epsilon}(2x+y)^3).
\]
As shown in \cite{S}, the origin is a stable fixed point for mapping $T$ if, and only if, $-\frac{1}{\epsilon}<0$; thus $\epsilon$ should be positive.

To locate the non--wandering set of mapping (\ref{2dmap}), we employ the results of \cite{LM1}. We adopt the notation used there, and rewrite our mapping as a difference equation with respect to the variable $x$:
\[
\frac{1}{\epsilon}x_{n+1}^3+x_n-2x_{n+1}+x_{n+2}=0.
\]
The leading term is $\frac{1}{\epsilon}x_{n+1}^3$, while the roots of polynomial $|\frac{1}{\epsilon}|x^3-4x$ are $0$ and $\pm 2\sqrt{|\epsilon|}$, thus the first coordinate of every point of the non--wandering set of (\ref{2dmap}) belongs to $[- 2\sqrt{|\epsilon}|,2\sqrt{|\epsilon}|]$. The second coordinate can be treated in exactly the same way.  
\end{proof}

\begin{figure}
\begin{center}
\includegraphics[scale=0.55]{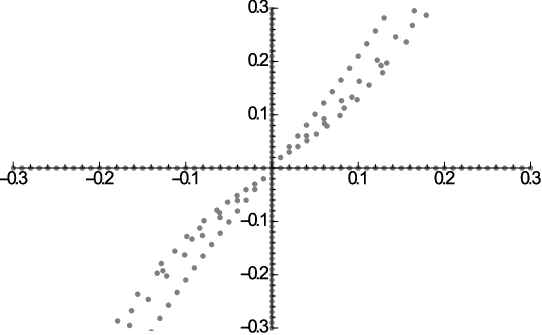}
\includegraphics[scale=0.55]{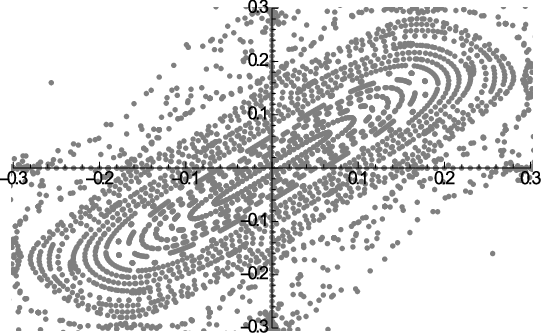}
\includegraphics[scale=0.36]{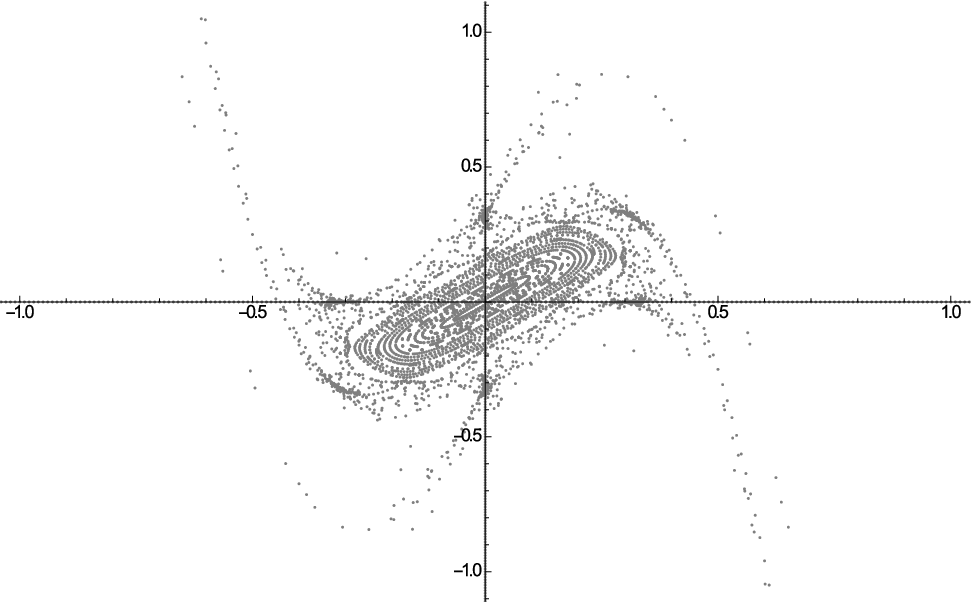}
\includegraphics[scale=0.36]{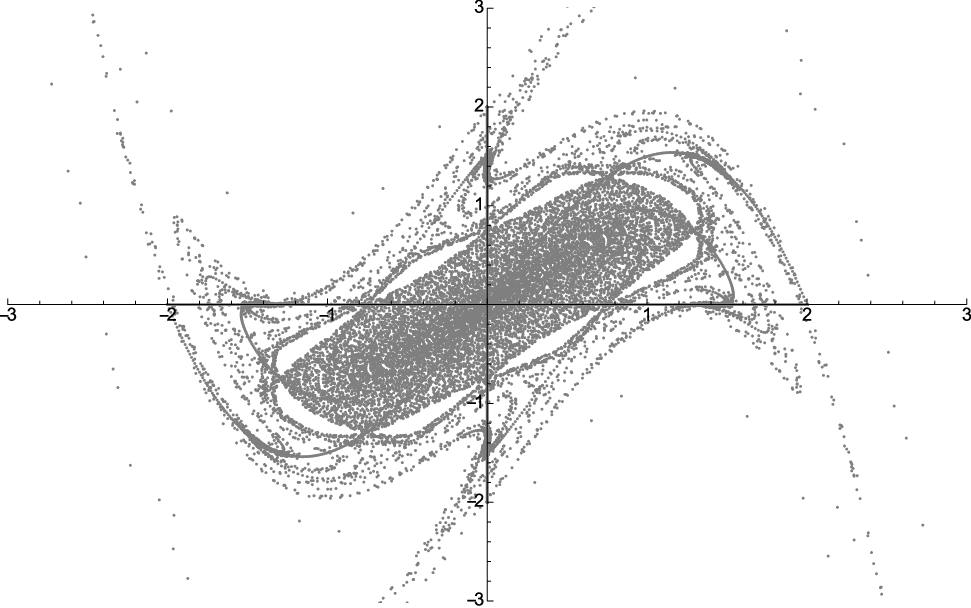}
\end{center}
\caption{Phase portraits for mapping (\ref{2dmap}). First row, left: $\epsilon =-0.1$ and all orbits, except from the fixed point, tend to infinity. First row, right: $\epsilon =0.1$ and the origin is a stable fixed point. Second row, left: $\epsilon=0.1$ global phase portrait. Second row, right: $\epsilon=2$ global phase portrait. The same initial conditions were used, for all four portraits.}
\end{figure}

In figure (1) we present the phase portrait of mapping (\ref{2dmap}), for specific values of the parameter $\epsilon$. We observe that, for $\epsilon <0$, all orbits, with the exception of the fixed point, tend to infinity, while for $\epsilon >0$, the origin is a stable fixed point. For different, positive, values of $\epsilon$, the structure of the phase portrait presents no qualitative change.

We shall now turn our attention to the case where $A\neq 0$, in equation (\ref{mainpdeeq}). 

\section{The $4$--d mapping}
Seeking steady--state solutions of equation (\ref{mainpdeeq}), we set $\dot{u}_n=0$ and $u_{n-2}=x,u_{n-1}=y,u_n=z,u_{n+1}=w$. We thus define the mapping:
\begin{center}
$f:\mathbb{C}^4\rightarrow \mathbb{C}^4,\ f(x,y,z,w)=(y,z,w,-\frac{1}{\epsilon A}|z|^2z-\frac{1}{A}w+\frac{2}{A}z-\frac{1}{A}y-x).$
\end{center}
Observe that, as above, $f(x,y,z,w)\in \mathbb{R}^4$ if, and only if, $(x,y,z,w)\in \mathbb{R}^4$. We therefore restrict our attention to:
\begin{equation}
\label{4dmap}
f:\mathbb{R}^4\rightarrow \mathbb{R}^4,\ f(x,y,z,w)=(y,z,w,-x-\frac{1}{A}y+\frac{2}{A}z-\frac{1}{\epsilon A}z^3-\frac{1}{A}w).
\end{equation}
It is this mapping that we wish to study, in this section.

This mapping is volume--preserving, since $|Df(x,y,z,w)|=1,\ \forall (x,y,z,w)\in \mathbb{R}^4$. Moreover, its inverse is also polynomial:
\begin{center}
$f^{-1}:\mathbb{R}^4\rightarrow \mathbb{R}^4,\ f^{-1}(x,y,z,w)=(-w-\frac{1}{A}x+\frac{2}{A}y-\frac{1}{\epsilon A}y^3-\frac{1}{A}z,x,y,z)$.
\end{center}
As in the $2$--d case, the following relations hold:
\begin{equation}
\label{symmetries}
(a)\ f \circ \sigma_4  = \sigma_4 \circ f,\ (b)\ f \circ \sigma_5  = \sigma_5 \circ f^{-1},\ (c)\ f \circ \sigma_6  = \sigma_6 \circ f^{-1}
\end{equation}
where:
$$\sigma_4(x,y,z,w)=(-x,-y,-z,-w),\ \sigma_5(x,y,z,w)=(w,z,y,x),$$
$$\sigma_6(x,y,z,w)=(-w,-z,-y,-x).$$

We locate the non--wandering set of mapping (\ref{4dmap}) as follows:
\begin{proposition}
The non--wandering set of mapping (\ref{4dmap}) is contained in
\[
\Big[-\sqrt{|\epsilon A|(2+\frac{4}{|A|})},\sqrt{|\epsilon A|(2+\frac{4}{|A|})}\Big]^4\subset \mathbb{R}^4.
\]
\end{proposition}
\begin{proof}
The result follows from \cite{LM1}. We rewrite mapping (\ref{4dmap}) as a difference equation, with respect to $x$:
\[
\frac{1}{\epsilon A}x_{n+2}^3+x_n+\frac{1}{A}x_{n+1}-\frac{2}{A}x_{n+2}+\frac{1}{A}x_{n+3}+x_{n+4}=0.
\]
The leading term is $\frac{1}{\epsilon A}x_{n+2}^3$, while the roots of polynomial $|\frac{1}{\epsilon A}|x^3-(2+\frac{4}{|A|})x$ are $0,\ \pm \sqrt{|\epsilon A|(2+\frac{4}{|A|})}$, thus the projection of the non--wandering set in the first axis is contained in $\Big[-\sqrt{|\epsilon A|(2+\frac{4}{|A|})},\sqrt{|\epsilon A|(2+\frac{4}{|A|})}\Big]$.

The inclusions concerning the projections on the other axes are proved in the same way.
\end{proof}

Let us now proceed to study the mappings fixed points.
\subsection{Fixed points and their stability}
In general, mapping $f$ possesses three fixed points, namely:
$$(0,0,0,0),\ (\pm \sqrt{-2\epsilon A},\pm \sqrt{-2\epsilon A},\pm \sqrt{-2\epsilon A},\pm \sqrt{-2\epsilon A}).$$
To study their stability, we shall use the following:
\begin{lemma}
Polynomial $p(x)=x^4+ax^3+bx^2+ax+1$, where $a,b\in \mathbb{R}$, possesses four real roots if, and only if, one of the following conditions is satisfied.
\begin{enumerate}
\item {$b<-2$ and $-\frac{1}{2}\sqrt{4+4b+b^2}<a<\frac{1}{2}\sqrt{4 + 4 b + b^2}$.}
\item {$b>6$ and $-\frac{1}{2}\sqrt{4+4b+b^2}<a<-\sqrt{-8 + 4 b}$.}
\item {$b>6$ and $\sqrt{-8 +4b}<a<\frac{1}{2}\sqrt{4+4b+b^2}$.}
\end{enumerate}
These roots are of the form $\lambda_1,\lambda_1^{-1},\lambda_2,\lambda_2^{-1}$, where $\lambda_1,\lambda_2\in (0,1)$. 
\end{lemma}
\begin{proof}
To determine for which values of $a,b,\ p(x)$ possesses real roots, we shall employ the Sturm theorem (see \cite{B}).

We construct the Sturm polynomials for $p(x)$, as follows:
\begin{align*}
p_0(x) &=p(x)=x^4+ax^3+bx^2+ax+1 \\
p_1(x) &=p'(x)=a+2bx+3ax^2+4x^3 \\
p_2(x) &=-1+\frac{a^2}{16}+(\frac{ab}{8}-\frac{3a}{4})x+(\frac{3 a^2}{16}-\frac{b}{2})x^2\\
p_3(x) &=-\frac{16(8+a^2-4b)(a(6-b)+6a^2x-2b(2+b)x)}{(3a^2-8 b)^2}\\
p_4(x) &=-\frac{(3 a^2 - 8 b)^2 (4 a^2 - (2 + b)^2)}{64 (-3 a^2 + b (2 + b))^2}.
\end{align*}
We remind that, as dictated by Sturm's theorem, $p_i(x),i\geq 2$, is defined to be the remainder of the division of $p_{i-2}(x)$ by $p_{i-1}(x)$, with its sign changed.

In order for $p(x)$ to possess $4$ real roots, it is sufficient and necessary to have:
$$\frac{3 a^2}{16}-\frac{b}{2}>0,\ (8+a^2-4b)(6a^2-2b(2+b))<0,\ (2+b)^2 -4a^2>0,$$
from which inequalities we deduce the conditions given in the statement of the lemma. 

To prove that these roots are of the form claimed above, note that mapping:
$$s\mapsto \frac{1+s}{1-s},$$
maps the left, open, half--plane of $\mathbb{C}$ to the closed unit disk. In the $s$--variable, polynomial $p(x)$ takes the following form (after multiplication with $(1-s)^4$):
$$2+2a+b-2(-6+b)s^2+(2-2a+b)s^4.$$
The Routh--Hurwitz criterion (see \cite{ST}), ensures that the roots of this polynomial appear in pairs and that members of the same pair are of equal norm, but opposite sign. Mapping $s \mapsto \frac{1+s}{1-s}$ maps such pairs to pairs of the form $\lambda,\lambda^{-1}$.   
\end{proof}

We are now able to state results concerning the stability of fixed points mapping (\ref{4dmap}) possesses.
\begin{corollary}
\begin{enumerate}\item[]
\item {The origin is always a fixed point for mapping (\ref{4dmap}). Its eigenvalues are:
\begin{itemize}
\item {two pairs of complex--conjugate numbers, for $A\in (-\infty <\frac{-2+\sqrt{2}}{4})\cup (2,+\infty)$}
\item {all real, for $A\in [\frac{-2+\sqrt{2}}{4},0)$.}
\item {two real and a pair of complex--conjugate eigenvalues, for $A\in (0,2]$.}
\end{itemize}}
\item {For $\epsilon A<0$, mapping $f$ possesses two more fixed points, symmetric to each other, namely $(\pm \sqrt{-2\epsilon A},\pm \sqrt{-2\epsilon A},\pm \sqrt{-2\epsilon A},\pm \sqrt{-2\epsilon A})$. Their eigenvalues are:
\begin{itemize}
\item {two real and a pair of complex--conjugate numbers, for $A\in (-1,0)$.}
\item {all real, for $A\in (-\infty,-1]\cup (0,+\infty)$.}
\end{itemize}}
\item {In any case, the eigenvalues are necessarily of the form $\lambda_1,\lambda_1^{-1},\lambda_2,\lambda_2^{-1}$.}
\end{enumerate}
\end{corollary} 
\begin{proof}
\begin{enumerate}\item[]
\item {The characteristic polynomial of $f$ at the origin is:
\[
k_0(x)=x^4+\frac{1}{A}x^3-\frac{2}{A}x^2+\frac{1}{A}x+1.
\]
It is obtained from polynomial $p(x)$, of the previous lemma, for $a=\frac{1}{A}$ and $b=-\frac{2}{A}$, thus the conclusion. The discriminant of this polynomial equals
\[
\frac{4 (-2 - 31 A - 144 A^2 - 176 A^3 + 64 A^5)}{A^5}.
\]
The conclusions follow from the lemma above and the properties of discriminants of polynomials (see \cite{B}).}

\item {The characteristic polynomial of $f$ at the non--trivial fixed points is:
\[
k_1(x)=x^4+\frac{1}{A}x^3-(6+\frac{2}{A})x^2+\frac{1}{A}x+1.
\]
It is obtained from the polynomial $p(x)$ of the lemma above, for $a=\frac{1}{A},b=-(6+\frac{2}{A})$. Its discriminant equals:
\[
\frac{16 (1 + 17 A + 144 A^2 + 640 A^3 + 1536 A^4 + 1024 A^5)}{A^5}.\]
Hence the conclusion.}
\end{enumerate}
\end{proof}

Since our main intention is to steady--state solitons for equation (\ref{mainpdeeq}), we turn our attention to orbits homoclinic to the origin, for mapping (\ref{4dmap}). 
\subsection{Homoclinics to the origin}

To locate solutions homoclinic to the origin, we restrict our attention to the case where all eigenvalues are real. Thus, according to the results presented above, we let $\ A\in [\frac{-2+\sqrt{2}}{4},0)$. We also let $\epsilon \in [-1,1]$ and, in view of section 2, we expect that homoclinics should exist for positive values of $\epsilon$. 

For these parameter values, the origin is a saddle point for mapping (\ref{4dmap}). It is easy to verify that $\pm 1$ are not eigenvalues of $Jf(0,0,0,0)$; thus the Jacobian matrix possesses two eigenvalues of absolute value less than $1$, which we shall denote by $\lambda_1,\ \lambda_2$, and two eigenvalues of absolute values greater than $1$, which we shall denote by $\lambda_3=\lambda_1^{-1},\ \lambda_4=\lambda_2^{-1}$.

The origin is therefore a hyperbolic saddle, having a two--dimensional stable and a two--dimensional unstable manifold, denoted by $W^s(O),W^u(O)$ respectively. It is our purpose here to determine if these two manifolds intersect.

To achieve that, we employ the Parametrization Method. We shall briefly sketch this method here, as applied in our case, however the interested reader should consult \cite{M1,M2,M3,M4,M5,M6} for applications of this method to the study of maps and \cite{HCF} for many more applications. 

Manifolds $W^s(O),W^u(O)$ are analytic, according to the Stable Manifold Theorem (see, for example, \cite{ST}) and can therefore be represented, at least locally, using analytic mappings. That is, the stable and the unstable manifolds can be thought of as the images of mappings $S^s,S^u:\mathbb{R}^2\rightarrow \mathbb{R}^4$, respectively. 

Restricted on the stable (unstable) manifold, mapping (\ref{4dmap}) acts as a contraction (dilation), with contraction (dilation) rates $\lambda_1,\lambda_2$ ($\lambda_3,\lambda_4$). Therefore, one gets the following equations:
\begin{eqnarray}
f\circ S^s(u,v)&=&S^s(\lambda_1u,\lambda_2v) \label{homeq1} \\ 
f\circ S^u(u,v)&=&S^u(\lambda_3u,\lambda_4v) \label{homeq2}.
\end{eqnarray}

The symmetry properties (\ref{symmetries}) of mapping (\ref{4dmap}) are reflected on these two manifolds. Indeed, note that:
\begin{equation*}
\begin{split}
f\circ S^s(u,v)&=S^s(\lambda_1u,\lambda_2v) \Rightarrow\\
\Rightarrow \sigma_4 \circ f\circ S^s(u,v)&=\sigma_4 \circ S^s(\lambda_1u,\lambda_2v)\Rightarrow \\
\xRightarrow[]{\ref{symmetries}(a)}f\circ \sigma_4 \circ S^s(u,v)&=\sigma_4 \circ S^s(\lambda_1u,\lambda_2v),
\end{split}
\end{equation*}
meaning that $\sigma _4 \circ S^s$ is also a stable manifold of the origin. Thus, the image of $S^s$ is symmetric, with respect to $\sigma_4$. The same holds for the image of $S^u$. 

Note also that:
\begin{equation*}
\begin{split}
f\circ S^s(u,v)&=S^s(\lambda_1u,\lambda_2v) \Rightarrow\\
\Rightarrow \sigma_5 \circ f\circ S^s(u,v)&=\sigma_5 \circ S^s(\lambda_1u,\lambda_2v)\Rightarrow \\
\xRightarrow[]{\ref{symmetries}(b)}f^{-1}\circ \sigma_5 \circ S^s(u,v)&=\sigma_5 \circ S^s(\lambda_1u,\lambda_2v),
\end{split}
\end{equation*}
Thus, $\sigma_5 \circ S^s$ is a stable manifold for $f^{-1}$ and therefore an unstable manifold for $f$; that is, $S^u$ is the image of $S^s$ under the symmetry $\sigma_5$ and the same of course holds for symmetry $\sigma _6$.

We begin by computing the stable manifold of the origin. Since this manifold is analytic, it can be represented as a, convergent, power series:
\[
S^s:\mathbb{R}^2\rightarrow \mathbb{R}^4,\ S^s(u,v)=
\begin{bmatrix}
\sum \limits _{n=0}^{+\infty}\sum \limits _{m=0}^{+\infty}a_1^{nm}u^nv^m \\
\sum \limits _{n=0}^{+\infty}\sum \limits _{m=0}^{+\infty}a_2^{nm}u^nv^m \\
\sum \limits _{n=0}^{+\infty}\sum \limits _{m=0}^{+\infty}a_3^{nm}u^nv^m \\
\sum \limits _{n=0}^{+\infty}\sum \limits _{m=0}^{+\infty}a_4^{nm}u^nv^m
\end{bmatrix}.
\]
In the notation above, $a_i^{nm},\ i=1,..,4$, stands for the $i$-th coefficient of order $n+m$; that is, the exponent is not a power.

The left--hand side of equation (\ref{homeq1}) reads as:
\[
\begin{bmatrix}
\sum \limits _{n=0}^{+\infty}\sum \limits _{m=0}^{+\infty}a_2^{nm}u^nv^m \\
\sum \limits _{n=0}^{+\infty}\sum \limits _{m=0}^{+\infty}a_3^{nm}u^nv^m \\
\sum \limits _{n=0}^{+\infty}\sum \limits _{m=0}^{+\infty}a_4^{nm}u^nv^m \\
\begin{multlined}
-\sum \limits _{n=0}^{+\infty}\sum \limits _{m=0}^{+\infty}a_1^{nm}u^nv^m-\frac{1}{A}\sum \limits _{n=0}^{+\infty}\sum \limits _{m=0}^{+\infty}a_2^{nm}u^nv^m+\frac{2}{A}\sum \limits _{n=0}^{+\infty}\sum \limits _{m=0}^{+\infty}a_3^{nm}u^nv^m -\\ -\frac{1}{\epsilon A}(\sum \limits _{n=0}^{+\infty}\sum \limits _{m=0}^{+\infty}a_3^{nm}u^nv^m)^3-\frac{1}{A}\sum \limits _{n=0}^{+\infty}\sum \limits _{m=0}^{+\infty}a_4^{nm}u^nv^m
\end{multlined}
\end{bmatrix},
\]
while the right--hand side equals:
\[
\begin{bmatrix}
\sum \limits _{n=0}^{+\infty}\sum \limits_{m=0}^{+\infty}\lambda_1^n\lambda_2^m a_1^{nm}u^nv^m \\
\sum \limits _{n=0}^{+\infty}\sum \limits _{m=0}^{+\infty}\lambda_1^n\lambda_2^ma_2^{nm}u^nv^m \\
\sum \limits _{n=0}^{+\infty}\sum \limits _{m=0}^{+\infty}\lambda_1^n\lambda_2^ma_3^{nm}u^nv^m \\
\sum \limits _{n=0}^{+\infty}\sum \limits _{m=0}^{+\infty}\lambda_1^n\lambda_2^ma_4^{nm}u^nv^m 
\end{bmatrix}.
\]  
Since:
$$(\sum _{n=0}^{+\infty}\sum _{m=0}^{+\infty}a_3^{nm}u^nv^m)^3=\sum_{n=0}^{+\infty}\sum_{m=0}^{+\infty}\sum_{i=0}^n\sum_{j=0}^m\sum_{k=0}^i\sum_{l=0}^ja_3^{n-i,m-j}a_3^{i-k,j-l}a_3^{k,l}u^nv^m,$$
by equating terms of the same degree, we arrive at the following system:
\begin{equation}
\label{homsystem}
\begin{split}
-\lambda_1^n\lambda_2^ma_1^{nm}+a_2^{nm}&=0 \\
-\lambda_1^n\lambda_2^ma_2^{nm}+a_3^{nm}&=0 \\
-\lambda_1^n\lambda_2^ma_3^{nm}+a_4^{nm}&=0 \\
-a_1^{nm}-\frac{1}{A}a_2^{nm}+\frac{2}{A}a_3^{nm}-(\frac{1}{A}+\lambda_1^n\lambda_2^m)a_4^{nm}&= \\
=\frac{1}{\epsilon A}\sum _{i=0}^{n}\sum_{j=0}^{m}\sum_{k=0}^i\sum_{l=0}^ja_3^{n-i,m-j}a_3^{i-k,j-l}a_3^{k,l}
\end{split}
\end{equation}

Note that system (\ref{homsystem}) is a linear system, with respect to the unknowns $a_i^{nm},\ i=1,..,4$. It can be used to find the coefficients of parametrization $S^s$ of arbitrary, finite, order, as long as the lower--order coefficients are known and the eigenvalues $\lambda_i,i=1,2$ satisfy a non--resonance condition.

For example, for $n=m=0$, system (\ref{homsystem}) becomes:
\begin{eqnarray*}
-a_1^{00}+a_2^{00}&=0 \\ 
-a_2^{00}+a_3^{00}&=0 \\
-a_3^{00}+a_4^{00}&=0 \\
-a_1^{00}-\frac{1}{A}a_2^{00}+\frac{2}{A}a_3^{00}-\frac{1}{A}a_4^{00}&=0
\end{eqnarray*}
from which we conclude that $a_i^{00}=0,\ i=1,..,4$. This reflects the fact that the stable manifold of the origin contains the origin; thus parametrization $S^s$ should satisfy $S^s(0,0)=(0,0,0,0)$.

Similarly, for $n=1,\ m=0$ and $n=0,\ m=1$ we get:
\[(a_1^{10},a_2^{10},a_3^{10},a_4^{10})=w_1,\  (a_1^{01},a_2^{01},a_3^{01},a_4^{01})=w_2,
\]
where $w_1,w_2$ are the eigenvectors corresponding to the eigenvalues $\lambda_1,\lambda _2$. This reflects the fact that the tangent space at the origin of the stable manifold should be the stable eigenspace of the origin, i.e., $T_0(Im(S^s))=E^s(0)$ and therefore:
\[
\frac{\partial S^s}{\partial u}(0,0)=w_1,\ 
\frac{\partial S^s}{\partial v}(0,0)=w_2.
\]
The low--order terms are therefore known and one can proceed to the computation of higher--order terms, using system (\ref{homsystem}), recursively.

The unstable manifold $S^u$ can be computed in exactly the same way, the only difference being that eigenvalues $\lambda_1,\ \lambda_2$ should be replaced with the unstable ones $\lambda_3,\ \lambda_4$.

Since, however, the coefficients of both the stable and the unstable manifold can be computed to a finite order, we end up with a polynomial approximation $P^s(u,v)$ for the stable manifold $S^s(u,v)$ and a polynomial approximation $P^u(u,v)$ of the unstable manifold $S^u(u,v)$.

To locate homoclinic points, one should solve equation:
\begin{equation}
\label{homoclinicseq}
P^u(u_1,v_1)=P^s(u_2,v_2).
\end{equation}
Solutions $(u_1,v_1,u_2,v_2)$ correspond to homoclinic points of mapping (\ref{4dmap}). These points can be represented both as $P^u(u_1,v_1)$ and as $P^s(u_2,v_2)$.

We intend to solve equation (\ref{homoclinicseq}) numerically. Therefore, its solutions are approximations of analytical solutions. A measure of the validity of these approximations is the norm $\|P^u(u_1,v_1)-P^s(u_2,v_2)\|$. One should accept only the approximate solutions for which this norm is smaller than a predetermined value.

Here, we proceeded as follows:
\begin{enumerate}
\item {Initial values for parameters $(\epsilon,A)$ where chosen. Recall that $\epsilon$ should belong to the interval $[-1,1]$, while parameter $A$ belongs to the interval $[\frac{-2+\sqrt{2}}{4},0)$, to ensure that all eigenvalues of the fixed point at the origin are real.}

\item {For these parameter values, we solved system (\ref{homsystem}) for the coefficients of the (un)stable manifold. It was found that the lowest bound for the error $\|P^u(u_1,v_1)-P^s(u_2,v_2)\|$ was achieved when we computed terms up to order $80$. Terms of higher order did not improve the error of our calculations.}

\item {We numerically solved equation (\ref{homoclinicseq}). The solutions, with respect to unknowns $u_1,v_1,u_2,v_2$, should satisfy $\|P^u(u_1,v_1)-P^s(u_2,v_2)\|<10^{-10}$. Numerical solutions which did not fulfil this criterion were rejected.}

\item {If equation (\ref{homoclinicseq}) had a solution fulfilling the criterion above, we concluded that an orbit homoclinic to the origin existed. If no such solution was found, the existence of a homoclinic orbit could not be established.}

\item {We then proceeded to other values of parameters $(\epsilon,A)$.}
\end{enumerate}
The procedure described above permits us to draw the following:

\vspace*{0.3cm}
\textbf{Main numerical result:} For mapping (\ref{4dmap}), orbits homoclinic to the origin exist for all $\epsilon \in (0,1],\ A\in [-0.145,-0.115]$.

\vspace*{0.3cm}
A few observations are in order here.

First, for $\epsilon \in [-1,0)$ no homoclinics were found. This is in accordance with the two--dimensional case, where the fixed point at the origin was stable if, and only if, $\epsilon >0$.

Second, for positive values of $\epsilon$, numerical solutions of equation (\ref{homoclinicseq}) were found, even for values of $A$ not belonging in $[-0.145,-0.115]$. The error bound at these solutions was, however, greater than $10^{-10}$ and they were therefore rejected.

Third, the error bound for the numerical solutions that we accepted was lower for values of $\epsilon$ closer to $0$. The error bound was increasing, for $\epsilon \rightarrow 1$. Concerning parameter $A$, the error bound was found to be minimal in the middle of interval $[-0.145,-0.115]$.    

We should note here that the numerical result presented above implies that equation (\ref{mainpdeeq}) possesses steady--state solitons for $\epsilon>0$ and $A\in [-0.145,-0.115]$. One can, approximately, construct these solitons using the $x$ coordinate of the homoclinics found here.
  
\vspace*{0.3cm}
\textbf{Illustrative example:} Set $A=-0.125,\ \epsilon=0.0004$ and use system (\ref{homsystem}) to compute $P^u,P^s$ up to terms of order $80$. Equation (\ref{homoclinicseq}) possesses solution:
\[
\begin{bmatrix}
u_1'\\
v_1'\\
u_2'\\
v_2'
\end{bmatrix}
=
\begin{bmatrix}
-0.6774898840489101\\
0.1347239986358392\\
0.1347239986358898\\
0.6774898840494726
\end{bmatrix},
\]
which solution corresponds to the homoclinic point: 
\[
P^u(u_1',v_1')=P^s(u_2',v_2')=
\begin{bmatrix}
9.23324715725\cdot 10^{-3}\\
1.32738452775\cdot 10^{-2}\\
1.32738452775\cdot 10^{-2}\\
9.23324715725\cdot 10^{-3}
\end{bmatrix}.
\]
The errors involved in these calculations are found to be of order $10^{-13}$.

In figure $2$, left, the projection on the $x-y$ plane of the orbit of $P^u(u_1',v_1')$ is shown, along with its symmetric orbit. They converge to the origin for both positive and negative iteration. In the same figure, right, the $x$--coordinate of the orbit is shown. It corresponds to a steady--state soliton of equation (\ref{mainpdeeq}).

\begin{figure}
\begin{center}
\includegraphics[scale=0.6]{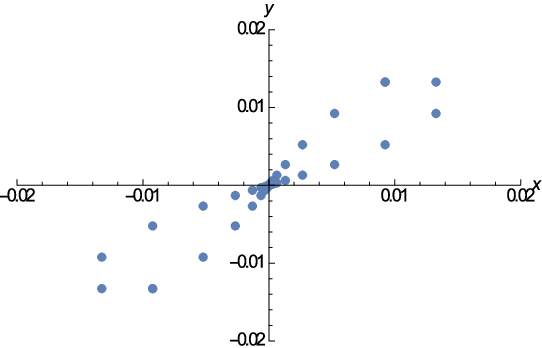}
\includegraphics[scale=0.6]{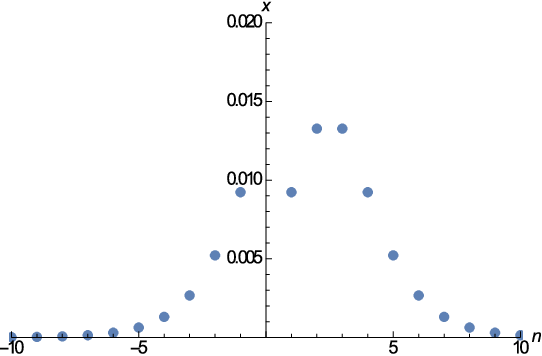}
\end{center}
\caption{Left panel: the projection, on the $x-y$ plane of a pair of symmetric homoclinic orbits for mapping (\ref{4dmap}). Their calculation is presented in the example below. Right panel: the projection of a homoclinic orbit on the $x$--axis: it corresponds to a steady state soliton of equation (\ref{mainpdeeq}).}
\end{figure}

Next, we wish to verify that the homoclinic points located here are transverse.
\subsection{Transversality of homoclinic points}

Let as consider a point $p\in W^u(O)\cap W^s(O)$. The intersection of $W^u(O),\ W^s(O)$ is transversal at $p$, if $T_pW^u(O)\oplus T_pW^s(O)\simeq \mathbb{R}^4$. Since the invariant manifolds of the origin can be represented in parametric form $S^u,\ S^s$, as we saw above, the transversality condition can be expressed as:
\begin{equation}
\label{determinant}
\det{ 
\
\begin{bmatrix}
\frac{\partial }{\partial u_1}S^u(u_1,v_1) \\
\frac{\partial }{\partial v_1}S^u(u_1,v_1) \\
\frac{\partial }{\partial u_2}S^s(u_2,v_2) \\
\frac{\partial }{\partial v_2}S^s(u_2,v_2) 
\end{bmatrix}
}
\neq 0,
\end{equation}
where the determinant is evaluated at $(u_1,v_1,u_2,v_2)$ corresponding to the point $p$, i.e. $S^u(u_1,v_1)=S^s(u_2,v_2)=p$.

Here, we use the expansions $P^s,\ P^u$ obtained above, and compute, for $A\in [-0.145,-0.115]$ and $\epsilon \in (0,1]$, the determinant appearing in the left--hand side of equation (\ref{determinant}) at the approximate solutions calculated in the previous section. We found that parameter $A$ has a strong influence on the value of the determinant, while the influence of parameter $\epsilon$ is indeed negligible.

In figure (3) we present the case where $\epsilon=0.0002$. The value of the determinant is found to be non--zero, implying the transverse intersection of the invariant manifolds of the origin, for every value of $A\in [-0.145,-0.115]$.     

\begin{figure}
\begin{center}
\includegraphics[scale=0.5]{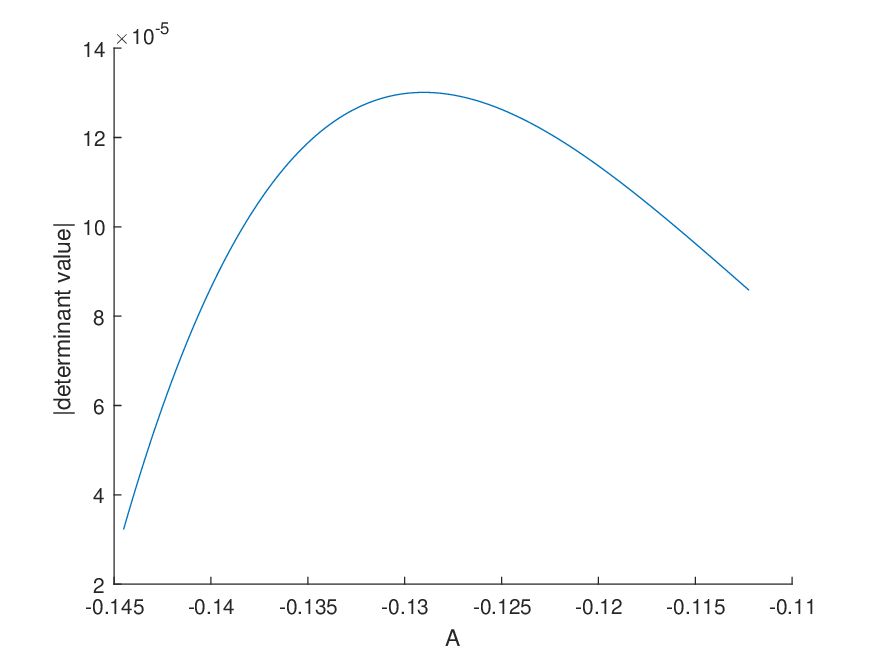}
\end{center}
\caption{The value of the determinant (\ref{determinant}), as a function of parameter $A$, for $\epsilon=0.0002$. It is non--zero, for every value of $A$, implying the transverse intersection of the invariant manifolds of the origin.}
\end{figure}

To further study the effect parameter $A$ has on the transversality of the intersection points, we computed the function which best fits the graph in figure (3). This function was found to be:
\[
T(A)=80.64A^4+33.74A^3+5.272A^2+0.3682A+0.009737.
\] 
This function vanishes at points $A=-0.146292$ and $A=-0.0891431$, implying that homoclinic points exist even when $A$ does not belong in the interval $[-0.145,-0.115]$ found above. This is probable due to the error bound used in our calculations: less strict error bounds would enlarge the existence interval of solutions for equation (\ref{homoclinicseq}).

Indeed, by relaxing condition $\|P^u(u_1,v_1)-P^s(u_2,v_2)\|<10^{-10}$ and allowing even bigger errors, we were able to locate homoclinic points for every $\epsilon >0,\ A\in [\frac{-2+\sqrt{2}}{4},0)$, which we chose not to present here.
\section{Conclusions}  
We studied a discrete NLS equation, with non--nearest neighbour interactions. This equation depends on two parameters, called $\epsilon$ and $A$. The equation, for $A=0,$ gives rise to a volume--preserving diffeomorphism of the plane, the behaviour of which turns out to be quite simple. For $A\neq 0$, however, seeking steady--state solutions of this equation, we get a four--dimensional diffeomorphism, the dynamical behaviour of which we analysed. We found that, for positive values of $\epsilon$ and for $A$ belonging in an interval of the negative semi--axis, this mapping possesses orbits homoclinic to the origin. Each such orbit corresponds to a steady--state discrete soliton of our discrete NLS--equation. Our techniques permit one not only to argue about the existence of these solitons but to also construct them, with great accuracy.

\end{document}